\documentclass[11pt, reqno]{amsart}

\usepackage{amsmath, amsthm, amscd, amsfonts, amssymb, graphicx, color}
\usepackage[bookmarksnumbered, colorlinks, plainpages]{hyperref}
\input{mathrsfs.sty}
\newtheorem{theorem}{Theorem}[section]
\newtheorem{lemma}[theorem]{Lemma}
\newtheorem{proposition}[theorem]{Proposition}
\newtheorem{corollary}[theorem]{Corollary}
\theoremstyle{definition}
\newtheorem{definition}[theorem]{Definition}
\newtheorem{example}[theorem]{Example}

\theoremstyle{remark}
\newtheorem{remark}[theorem]{Remark}
\numberwithin{equation}{section}

\begin{document}

\title[Exact and approximate operator parallelism]{Exact and approximate operator parallelism}

\author[M.S. Moslehian, A. Zamani]{Mohammad Sal Moslehian and Ali Zamani}

\address{Department of Pure Mathematics, Center of Excellence in
Analysis on Algebraic Structures (CEAAS), Ferdowsi University of
Mashhad, P.O. Box 1159, Mashhad 91775, Iran.}
\email{moslehian@um.ac.ir; moslehian@member.ams.org; zamani.ali85@yahoo.com}


\begin{abstract}
Extending the notion of parallelism we introduce the concept of approximate parallelism in normed spaces and then substantially restrict ourselves to the setting of Hilbert space operators endowed with the operator norm. We present several characterizations of the exact and approximate operator parallelism in the algebra $\mathbb{B}(\mathscr{H})$ of bounded linear operators acting on a Hilbert space $\mathscr{H}$. Among other things, we investigate the relationship between approximate parallelism and norm of inner derivations on $\mathbb{B}(\mathscr{H})$. We also characterize the parallel elements of a $C^*$-algebra by using states. Finally we utilize the linking algebra to give some equivalence assertions regarding parallel elements in a Hilbert $C^*$-module.
\end{abstract} \maketitle

\section{Introduction and preliminaries}
Let $\mathscr A$ be a $C^*$-algebra. An element $a\in\mathscr A$ is called \emph{positive} (we write $a\geq0$) if $a=b^* b$ for some $b\in\mathscr A$. If $a\in\mathscr A$ is positive, then exists a unique positive element $b\in\mathscr A$ such that $a=b^2$. Such an element $b$ is called \emph{the positive square root} of $a$. A linear functional $\varphi$ over $\mathscr A$ of norm one is called \emph{state} if $\varphi(a)\geq0$ for any positive element $a\in\mathscr A$. By $S(\mathscr A)$ we denote the set of all states of $\mathscr A$.\\
Throughout the paper, $\mathbb{K}(\mathscr{H})$ and $\mathbb{B}(\mathscr{H})$ denote the $C^*$-algebra of all compact operators and the $C^*$-algebra of all bounded linear operators on a complex Hilbert space $\mathscr{H}$ endowed with an inner product $(.\mid.)$, respectively. We stand $I$ for the identity operator on $\mathscr{H}$. Furthermore, For $\xi, \eta\in \mathscr{H}$, the rank one operator $\xi\otimes\eta$ on $\mathscr{H}$ is defined by $(\xi\otimes\eta)(\zeta)=(\zeta\mid\eta)\xi$. Note that by the Gelfand--Naimark theorem we can regard $\mathscr A$ as a $C^*$-subalgebra of $\mathbb{B}(\mathscr{H})$ for a complex Hilbert space $\mathscr{H}$. More details can be found e.g. in \cite{dix, mor}.\\
The notion of Hilbert $C^*$-module is a natural generalization of that of Hilbert space arising under replacement of the field of scalars $\mathbb{C}$ by a $C^*$-algebra. This concept plays a significant role in the theory of operator algebra and $K$-theory; see \cite{Man}.
Let $\mathscr A$ be a $C^*$-algebra. An inner product $\mathscr A$-module is a complex linear space $\mathscr{X}$ which is a right $\mathscr A$-module
with a compatible scalar multiplication (i.e., $\mu(xa) = (\mu x)a = x(\mu a)$ for all $x\in \mathscr{X}, a\in\mathscr A, \mu\in\mathbb{C}$) and equipped with an $\mathscr A$-valued inner product $\langle\cdot,\cdot\rangle \,: \mathscr{X}\times \mathscr{X}\longrightarrow\mathscr A$ satisfying\\
(i) $\langle x, \alpha y+\beta z\rangle=\alpha\langle x, y\rangle+\beta\langle x, z\rangle$,\\
(ii) $\langle x, ya\rangle=\langle x, y\rangle a$,\\
(iii) $\langle x, y\rangle^*=\langle y, x\rangle$,\\
(iv) $\langle x, x\rangle\geq0$ and $\langle x, x\rangle=0$ if and only if $x=0$,\\
for all $x, y, z\in \mathscr{X}, a\in\mathscr A, \alpha, \beta\in\mathbb{C}$. For an inner product $\mathscr A$-module $\mathscr{X}$ the Cauchy--Schwarz inequality holds (see \cite{MOS2} and references therein):
$$\|\langle x, y\rangle\|^2\leq\|\langle x, x\rangle\|\,\|\langle y, y\rangle\|\qquad(x, y\in \mathscr{X}).$$
Consequently, $\|x\|=\|\langle x, x\rangle\|^{\frac{1}{2}}$ defines a norm on $\mathscr{X}$. If $\mathscr{X}$ with respect to this norm is complete, then it is called a \emph{Hilbert $\mathscr A$-module}, or a \emph{Hilbert $C^*$-module} over $\mathscr A$. Complex Hilbert spaces are Hilbert $\mathbb{C}$-modules. Any $C^*$-algebra $\mathscr A$ can be regarded as a Hilbert $C^*$-module over itself via $\langle a, b\rangle :=a^* b$. For every $x\in \mathscr{X}$ the positive square root of $\langle x, x\rangle$ is denoted by $|x|$. If $\varphi$ be a state over $\mathscr A$, we have the following useful version of the Cauchy–-Schwarz inequality:
$$\varphi(\langle y, x\rangle)\varphi(\langle x, y\rangle)=|\varphi(\langle x, y\rangle)|^2\leq\varphi(\langle x, x\rangle)\varphi(\langle y, y\rangle)$$
for all $x, y\in \mathscr{X}$.\\
Let $\mathscr{X}$ and $\mathscr{Y}$ be two Hilbert $\mathscr A$-modules. A mapping $T\,:\mathscr{X}\longrightarrow \mathscr{Y}$ is called \emph{adjointable} if there exists a mapping
$S\,:\mathscr{Y}\longrightarrow \mathscr{X}$ such that $\langle Tx, y\rangle=\langle x, Sy\rangle$ for all $x\in \mathscr{X}, y\in \mathscr{Y}$. The unique mapping $S$ is denoted by $T^*$ and
is called the \emph{adjoint} of $T$. It is easy to see that $T$ must be bounded linear $\mathscr A$-module mapping. The space $\mathbb{B}(\mathscr{X},\mathscr{Y})$ of all adjointable maps between Hilbert $\mathscr A$-modules $\mathscr{X}$ and $\mathscr{Y}$ is a Banach space while
$\mathbb{B}(\mathscr{X}) :=\mathbb{B}(\mathscr{X},\mathscr{X})$ is a $C^*$-algebra.
By $\mathbb{K}(\mathscr{X},\mathscr{Y})$ we denote the closed linear subspace of $\mathbb{B}(\mathscr{X},\mathscr{Y})$ spanned by $\{\theta_{y,x} : x\in \mathscr{X}, y\in \mathscr{Y}\}$,
where $\theta_{y,x}$ is defined by $\theta_{y,x}(z)=y\langle x, z\rangle$. Elements of $\mathbb{K}(\mathscr{X},\mathscr{Y})$ are often referred to as ``compact'' operators. We write $\mathbb{K}(\mathscr{X})$ for $\mathbb{K}(\mathscr{X},\mathscr{X})$.\\
Any Hilbert $\mathscr A$-module can be embedded into a certain $C^*$-algebra. To see this, let $\mathscr{X}\oplus \mathscr A$ be the direct sum of the Hilbert $\mathscr A$-modules $\mathscr{X}$ and $\mathscr A$ equipped with the $\mathscr A$-inner product $\langle (x, a), (y, b)\rangle=\langle x, y\rangle+a^* b$, for every $x, y\in \mathscr{X}, a, b\in\mathscr A$. Each $x\in \mathscr{X}$ induces the maps $r_x\in\mathbb{B}(\mathscr A,\mathscr{X})$ and $l_x\in\mathbb{B}(\mathscr{X},\mathscr A)$ given by $r_x(a) = xa$ and $l_x(y)=\langle x, y\rangle$, respectively, such that $r^*_x=l_x$. The map $x\mapsto r_x$ is an isometric linear isomorphism of $\mathscr{X}$ to $\mathbb{K}(\mathscr A,\mathscr{X})$ and $x\mapsto l_x$ is an isometric conjugate linear isomorphism of $\mathscr{X}$ to $\mathbb{K}(\mathscr{X},\mathscr A)$. Further, every $a\in \mathscr A$ induces the map $T_a\in\mathbb{K}(\mathscr A)$ given by $T_a(b)=ab$. The
map $a\mapsto T_a$ defines an isomorphism of $C^*$-algebras $\mathscr A$ and $\mathbb{K}(\mathscr A)$. Set
$$\mathbb{L}(\mathscr{X})=\begin{bmatrix}
\mathbb{K}(\mathscr A) & \mathbb{K}(\mathscr{X},\mathscr A) \\
\mathbb{K}(\mathscr A,\mathscr{X}) & \mathbb{K}(\mathscr{X})
\end{bmatrix}=\left\{\begin{bmatrix}
T_a & l_y \\
r_x & T
\end{bmatrix} : a\in\mathscr A,\, x, y\in \mathscr{X},\, T\in\mathbb{K}(\mathscr{X})\right\}$$
Then $\mathbb{L}(\mathscr{X})$ is a $C^*$-subalgebra of $\mathbb{K}(\mathscr{X}\oplus \mathscr A)$, called the \emph{linking algebra} of $\mathscr{\mathscr{X}}$. Clearly
$$\mathscr{X}\simeq\begin{bmatrix}
0 & 0 \\
\mathscr{X} & 0
\end{bmatrix},\quad \mathscr A\simeq\begin{bmatrix}
\mathscr A & 0 \\
0 & 0
\end{bmatrix},\quad \mathbb{K}(\mathscr{X})\simeq\begin{bmatrix}
0 & 0 \\
0 & \mathbb{K}(\mathscr{X})
\end{bmatrix}.$$
Furthermore, $\langle x, y\rangle$ of $\mathscr{X}$ becomes the product $l_xr_y$ in $\mathbb{L}(\mathscr{X})$ and the module multiplication of $\mathscr{\mathscr{X}}$ becomes a part of the internal multiplication of $\mathbb{L}(\mathscr{X})$. We refer the reader to \cite{lan, RW} for more information on Hilbert $C^*$-modules and linking algebras.

Following Seddik \cite{se.1} we introduce a notion of parallelism in normed spaces in Section 2. Inspired by the approximate Birkhoff--James orthogonality ($\varepsilon$-orthogonality) introduced by Dragomir \cite{DRA} and a variant of $\varepsilon$-orthogonality given by Chmieli\'{n}ski \cite{J.C} which has been investigated by Ili\v{s}evi\'{c} and Turn\v{s}ek \cite{M.T} in the setting of Hilbert $C^*$-modules, we introduce a notion of
approximate parallelism ($\varepsilon$-parallelism).

In the next sections, we substantially restrict ourselves to the setting of Hilbert space operators equipped with the operator norm. In section 3, we present several characterizations of the exact and approximate operator parallelism in the algebra $\mathbb{B}(\mathscr{H})$ of bounded linear operators acting on a Hilbert space $\mathscr{H}$. Among other things, we investigate the relationship between approximate parallelism and norm of inner derivations on $\mathbb{B}(\mathscr{H})$. In Section 4, we characterize the parallel elements of a $C^*$-algebra by using states and utilize the linking algebra to give some equivalence assertions regarding parallel elements in a Hilbert $C^*$-module.

\section{Parallelism in normed spaces}

We start our work with the following definition of parallelism in normed spaces.
\begin{definition}
Let $\mathscr{V}$ be a normed space. The vector $x\in \mathscr{V}$ is \emph{exact parallel} or simply \emph{parallel} to $y\in \mathscr{V}$, denoted by $x\parallel y$ (see \cite{se.1}), if
\begin{eqnarray}\label{p}
\|x+\lambda y\|=\|x\|+\|y\|, \mbox{~for~some~ $\lambda\in\mathbb{T}=\{\alpha\in\mathbb{C}: \,\,|\alpha|=1\}$}.
\end{eqnarray}
\end{definition}
Notice that the parallelism is a symmetric relation. It is easy to see that if $x, y$ are linearly dependent, then $x\parallel y$. The converse is however not true, in general.
\begin{example} Let us consider the space $(\mathbb{R}^2, |||.|||)$ where $|||(x_1 , x_2)|||=\max\{|x_1|, |x_2|\}$ for all $(x_1, x_2)\in\mathbb{R}^2$. Let $x=(1, 0), y=(1, 1)$ and $\lambda=1$. Then $x, y$ are linearly independent and $|||x+\lambda y|||=|||(2, 1)|||=2=|||x|||+|||y|||$, i.e., $x\parallel y$.
\end{example}

An operator $T$ on a separable complex Hilbert space is said to be in the Schatten $p$-class $\mathcal{C}_p\,\,(1\leq p <\infty)$, if ${\rm tr}(|T|^p)< \infty$. The Schatten $p$-norm of $T$ is defined by $\|T\|_p = ({\rm tr}(|T|^p))^{\frac{1}{p}}$. For $1< p \leq 2$ and $\frac{1}{p} + \frac{1}{q} = 1$, the Clarkson inequality for $T, S\in \mathcal{C}_p$ asserts that
$$\|T + S\|^q_p + \|T - S\|^q_p \leq 2 (\|T\|^p_p + \|S\|^p_p)^{\frac{q}{p}},$$
which can be found in \cite{McCarthy}.
\begin{theorem}
Let $T, S\in \mathcal{C}_p$ with $1< p \leq 2$ and $\frac{1}{p} + \frac{1}{q} = 1$. The following statements are equivalent:\\
(i) $T, S$ are linearly dependent;\\
(ii) $T\parallel S$.
\end{theorem}
\begin{proof}
Obviously, (i)$\Longrightarrow$(ii).\\
Suppose (ii) holds. Therefore $\|T + \lambda S\|_p = \|T\|_p + \|S\|_p$ for some $\lambda\in\mathbb{T}$. Without loss of generality we may assume that $\|T\|_p \leq \|S\|_p$. We have
\begin{align*}
2\|T\|_p &= \|T\|_p + \left\|\frac{\lambda \|T\|_p}{\|S\|_p}S\right\|_p \geq\left\|T + \frac{\lambda \|T\|_p}{\|S\|_p}S\right\|_p
 = \left\|T + \lambda S - \lambda\left(1 - \frac{\|T\|_p}{\|S\|_p}\right)S\right\|_p
\\&\geq \|T + \lambda S\|_p - \left(1 - \frac{\|T\|_p}{\|S\|_p}\right)\|S\|_p
= \|T\|_p + \|S\|_p - \left(1 - \frac{\|T\|_p}{\|S\|_p}\right)\|S\|_p
= 2\|T\|_p,
\end{align*}
so $\left\|T + \frac{\lambda \|T\|_p}{\|S\|_p}S\right\|_p = 2\|T\|_p$. Hence by Clarkson inequality we get
\begin{align*}
2^q\|T\|^q_p + \left\|T - \frac{\lambda \|T\|_p}{\|S\|_p}B\right\|^q_p & = \left\|T + \frac{\lambda \|T\|_p}{\|S\|_p}S\right\|^q_p + \left\|T - \frac{\lambda \|T\|_p}{\|S\|_p}S\right\|^q_p
\\&\leq 2\left(\|T\|^p_p + \left\|\frac{\lambda \|T\|_p}{\|S\|_p}S\right\|^p_p\right)^{\frac{q}{p}}
= 2^{1 + \frac{q}{p}}\|T\|^q_p = 2^q\|T\|^q_p,
\end{align*}
wherefrom we get $\left\|T - \frac{\lambda \|T\|_p}{\|S\|_p}S\right\|^q_p = 0$. Hence $T = \frac{\lambda \|T\|_p}{\|S\|_p}S$, which gives (i).
\end{proof}

The following important example is the motivation for further discussion.
\begin{example}
If $\tau_1, \tau_2$ are positive linear functionals on a $C^*$-algebra ${\mathscr A}$. Then for $\lambda=1\in\mathbb{T}$, by \cite[Corollary 3.3.5]{mor} we have $\|\tau_1+\lambda\tau_2\|=\|\tau_1\|+\|\tau_2\|$. So $\tau_1\parallel \tau_2$.
\end{example}
\begin{example}
Suppose that $\tau$ is a self-adjoint bounded linear functional on a $C^*$-algebra. By Jordan Decomposition Theorem \cite[Theorem 3.3.10]{mor}, there exist positive linear functionals $\tau_+, \tau_-$ such that $\tau = \tau_+-\tau_-$ and $\|\tau\|=\|\tau_+\|+\|\tau_-\|$. Thus for $\lambda=-1\in\mathbb{T}$ we have $\|\tau_++\lambda\tau_-\|=\|\tau_+\|+\|\tau_-\|$. Hence $\tau_+\parallel \tau_-$.
\end{example}

For every $\varepsilon\in[0 , 1)$, the following notion of the approximate Birkhoff--James orthogonality ($\varepsilon$-orthogonality) was
introduced by Dragomir \cite{DRA} as
$$x\perp^\varepsilon y \Longleftrightarrow \|x+\lambda y\|\geq(1-\varepsilon)\|x\|\,\,\,\,(\lambda\in\mathbb{C}).$$
In addition, an alternative definition of $\varepsilon$-orthogonality was given by Chmieli\'{n}ski \cite{J.C}. These facts motivate us to give the following definition of approximate parallelism ($\varepsilon$-parallelism) in the setting of normed spaces.

\begin{definition}
Two elements $x$ and $y$ in a a normed space are \emph{approximate parallel ($\varepsilon$-parallel)}, denoted by $x\parallel^{\varepsilon} y$, if
\begin{eqnarray}\label{ap}
\inf\{\|x+\mu y\| :\, \mu\in\mathbb{C}\}\leq\varepsilon\|x\|.
\end{eqnarray}
\end{definition}
It is remarkable that the relation $\varepsilon$-parallelism for $\varepsilon=0$ is the same as the exact parallelism.
\begin{proposition}
In a normed space, the $0$-parallelism is the same as exact parallelism.
\end{proposition}
\begin{proof} Let us assume that $x\neq0$ and choose sequence $\{\mu_n\}$ of vectors in $\mathbb{C}$ such that $\lim_{n\rightarrow\infty} \|x+\mu_n y\|=0$.
It follows from $|\mu_n|\,\|y\|\leq\|x+\mu_n y\|+\|x\|$ that the sequence $\{\mu_n\}$ is bounded. Therefore there exists a subsequence $\{\mu_{k_n}\}$ which is convergent to a number $\mu_0$. Since $x\neq 0$ and $\lim_{n\rightarrow\infty} \|x+\mu_n y\|=0$, we conclude that $\mu_0\neq 0$ as well as $\|x+\mu_0 y\|=0$, or equivalently, $x=-\mu_0 y$. Put $\lambda=-\frac{|\mu_0|}{\overline{\mu_0}}\in\mathbb{T}$. Then
$$\|x+\lambda y\|=\left\|-\mu_0 y-\frac{|\mu_0|}{\overline{\mu_0}}y\right\|=(|\mu_0|+1)\|y\|=\|-\mu_0y\|+\|y\|=\|x\|+\|y\|,$$
whence $\|x+\lambda y\|=\|x\|+\|y\|$ for some $\lambda\in\mathbb{T}$, i.e., $x\parallel y$.\\
\end{proof}

From now on we deal merely with the space $\mathbb{B}(\mathscr{H})$ endowed with the operator norm.

\section{Operator parallelism}
In the present section, we discuss the exact and approximate operator parallelism. These notions can be defined by the same formulas as \eqref{p} and \eqref{ap} in normed spaces. Thus $$T_1\parallel T_2\Leftrightarrow \|T_1+\lambda T_2\|=\|T_1\|+\|T_2\|$$ for some $\lambda\in\mathbb{T}$. The following example shows that the concept of operator parallelism is important.
\begin{example}
Suppose that $T$ is a compact self-adjoint operator on a Hilbert space $\mathscr{H}$. Then either $\|T\|$ or $-\|T\|$ is an eigenvalue of $T$. We may assume that $\|T\|=1$ is an eigenvalue of $T$. Therefore there exists a nonzero vector $x\in\mathscr{H}$ such that $Tx=x$. Hence
$2\|x\|=\|(T+I)x\|\leq\|T+I\|\|x\|\leq2\|x\|$. So we get $\|T+I\|=2=\|T\|+\|I\|$. Thus $T\parallel I$ and $T$ fulfils the Daugavet equation $\|T+I\|=\|T\|+1$; see \cite{WER}. This shows that the Daugavet equation is closely related to the notion of parallelism.
\end{example}
In the following proposition we state some basic properties of operator parallelism.
\begin{proposition}\label{pr.11}
Let $T_1, T_2\in \mathbb{B}(\mathscr{H})$. The following statements are equivalent:\\
(i) $T_1\parallel T_2$;\\
(ii) $T_1^*\parallel T_2^*$;\\
(iii) $\alpha T_1\parallel \beta T_2$ \ \ \ \ \ $(\alpha, \beta\in\mathbb{R}\smallsetminus\{0\})$;\\
(iv) $\gamma T_1\parallel \gamma T_2$ \ \ \ \ \ $(\gamma\in\mathbb{C}\smallsetminus\{0\})$.

\end{proposition}
\begin{proof}
The equivalences (i)$\Leftrightarrow$(ii)$\Leftrightarrow$(iv) immediately follow from the definition of operator parallelism.\\
(i)$\Longrightarrow$(iii) Suppose that $\alpha, \beta\in\mathbb{R}\smallsetminus\{0\}$ and $T_1\parallel T_2$. Hence $\|T_1+\lambda T_2\|=\|T_1\|+\|T_2\|$ for some $\lambda\in\mathbb{T}$. We can assume that $\alpha \geq \beta>0$. We therefore have
\begin{align*}
\|\alpha T_1\|+\|\beta T_2\|&\geq\|\alpha T_1+\lambda(\beta T_2)\|
=\|\alpha(T_1+\lambda T_2)-(\alpha-\beta)(\lambda T_2)\|
\\&\geq \|\alpha(T_1+\lambda T_2)\|-\|(\alpha-\beta)\lambda T_2\|=\alpha\|T_1+\lambda T_2\|-(\alpha-\beta)\|T_2\|
\\&=\alpha(\|T_1\|+\|T_2\|)-(\alpha-\beta)\|T_2\|=\|\alpha T_1\|+\|\beta T_2\|,
\end{align*}
whence $\|\alpha T_1+\lambda(\beta T_2)\|=\|\alpha T_1\|+\|\beta T_2\|$ for some $\lambda\in\mathbb{T}$. So $\alpha T_1\parallel \beta T_2$.\\
(iii)$\Longrightarrow$(i) is obvious.
\end{proof}
In what follows $\sigma(T)$ and $r(T)$ stand for the spectrum and spectral radius of an arbitrary element $T\in \mathbb{B}(\mathscr{H})$, respectively. In the following theorem we shall characterize the operator parallelism.
\begin{theorem}\label{th.13}
Let $T_1, T_2\in \mathbb{B}(\mathscr{H})$. Then the following statements are equivalent:\\
(i) $T_1\parallel T_2$;\\
(ii) There exist a sequence of unit vectors $\{\xi_n\}$ in $\mathscr{H}$ and $\lambda\in\mathbb{T}$ such that $$\lim_{n\rightarrow\infty} (T_1\xi_n\mid T_2\xi_n)=\lambda\|T_1\|\,\|T_2\|;$$
(iii) $r(T_2^* T_1)=\|T_2^* T_1\|=\|T_1\|\,\|T_2\|$;\\
(iv) $T_1^* T_1\parallel T_1^* T_2$ and $\|T_1^* T_2\|=\|T_1\|\,\|T_2\|$;\\
(v) $\|T_1^*(T_1+\lambda T_2)\|=\|T_1\|(\|T_1\|+\|T_2\|)$ for some $\lambda\in\mathbb{T}$.
\end{theorem}
\begin{proof}
(i)$\Leftrightarrow$(ii) Let $T_1\parallel T_2$. Then $\|T_1+\lambda T_2\|=\|T_1\|+\|T_2\|$ for some $\lambda\in\mathbb{T}$. Since
$$\sup\{\|T_1\xi+\lambda T_2\xi\|:\, \xi\in\mathscr{H},\|\xi\|=1\}=\|T_1+\lambda T_2\|=\|T_1\|+\|T_2\|,$$
there exists a sequence of unit vectors $\{\xi_n\}$ in $\mathscr{H}$ such that $\lim_{n\rightarrow\infty} \|T_1\xi_n+\lambda T_2\xi_n\|=\|T_1\|+\|T_2\|$. We have
\begin{align*}
\|T_1\|^2+2\|T_1\|\,\|T_2\|+\|T_2\|^2&=\lim_{n\rightarrow\infty} \|T_1\xi_n+\lambda T_2\xi_n\|^2
\\&=\lim_{n\rightarrow\infty} \left[\,\|T_1\xi_n\|^2+(T_1\xi_n\mid \lambda T_2\xi_n)+(\lambda T_2\xi_n\mid T_1\xi_n)+\|T_2\xi_n\|^2\,\right]
\\&\leq\|T_1\|^2+2\lim_{n\rightarrow\infty} |(T_1\xi_n\mid \lambda T_2\xi_n)|+\|T_2\|^2
\\&\leq\|T_1\|^2+2\|T_1\|\,\|T_2\|+\|T_2\|^2.
\end{align*}
Therefore $\lim_{n\rightarrow\infty} (T_1\xi_n\mid \lambda T_2\xi_n)|=\|T_1\|\,\|T_2\|$, or equivalently,
$$\lim_{n\rightarrow\infty} (T_1\xi_n\mid T_2\xi_n)=\lambda\|T_1\|\,\|T_2\|.$$
To prove the converse, suppose that there exist a sequence of unit vectors $\{\xi_n\}$ in $\mathscr{H}$ and $\lambda\in\mathbb{T}$ such that $\lim_{n\rightarrow\infty} (T_1\xi_n\mid T_2\xi_n)=\lambda\|T_1\|\,\|T_2\|$. It follows from
$$\|T_1\|\,\|T_2\|=\lim_{n\rightarrow\infty} |(T_1\xi_n\mid T_2\xi_n)|\leq\lim_{n\rightarrow\infty} \|T_1\xi_n\|\,\|T_2\|\leq\|T_1\|\,\|T_2\|,$$
that $\lim_{n\rightarrow\infty} \|T_1\xi_n\|=\|T_1\|$ and by using a similar argument, $\lim_{n\rightarrow\infty} \|T_2\xi_n\|=\|T_2\|$.
So that
$$\lim_{n\rightarrow\infty} \mbox{Re}(T_1\xi_n\mid \lambda T_2\xi_n)=\lim_{n\rightarrow\infty} (T_1\xi_n\mid \lambda T_2\xi_n)=\|T_1\|\,\|T_2\|,$$
whence we reach
\begin{align*}
\|T_1\|+\|T_2\|&\geq\|T_1+\lambda T_2\|
\geq\left(\lim_{n\rightarrow\infty} \|T_1\xi_n+\lambda T_2\xi_n\|^2\right)^\frac{1}{2}
\\&=\left(\lim_{n\rightarrow\infty} \left[\|T_1\xi_n\|^2+2\mbox{Re}(T_1\xi_n\mid \lambda T_2\xi_n)+\| T_2\xi\|^2\right]\right)^\frac{1}{2}
\\&=\left(\|T_1\|^2+2\|T_1\|\,\|T_2\|+\| T_2\|^2\right)^\frac{1}{2}
=\|T_1\|+\|T_2\|.
\end{align*}
Thus $\|T_1+\lambda T_2\|=\|T_1\|+\|T_2\|$, so $T_1\parallel T_2$.\\
(ii)$\Leftrightarrow$(iii) Let $\{\xi_n\}$ be a sequence of unit vectors in $\mathscr{H}$ satisfying
$\lim_{n\rightarrow\infty} (T_1\xi_n\mid T_2\xi_n)=\lambda\|T_1\|\,\|T_2\|$,
for some $\lambda\in\mathbb{T}$. By the equivalence (i)$\Leftrightarrow$(ii) we have $T_1\parallel T_2$. Hence
\begin{align*}
(\|T_1\|+\|T_2\|)^2&=\|T_1+\lambda T_2\|^2=\|(T_1+\lambda T_2)^*(T_1+\lambda T_2)\|
\\&=\|T_1^* T_1+\lambda T_1^* T_2+\overline{\lambda} T_2^* T_1+T_2^* T_2\|
\\&\leq\|T_1^* T_1\|+\|\lambda T_1^* T_2\|+\|\overline{\lambda} T_2^* T_1\|+\|T_2^* T_2\|
\\&=\|T_1\|^2+2\|T_1\|\,\|T_2\|+\|T_2\|^2
\\&=(\|T_1\|+\|T_2\|)^2,
\end{align*}
so $\|T_1^* T_2\|=\|T_1\|\,\|T_2\|.$ \\
Since
\begin{align*}
\|T_1\|\,\|T_2\|=\lim_{n\rightarrow\infty} |(T_1\xi_n\mid T_2\xi_n)|
=\lim_{n\rightarrow\infty} |(T_2^* T_1\xi_n\mid \xi_n)|
\leq\lim_{n\rightarrow\infty} \|T_2^* T_1\xi_n\|
\leq\|T_2^*\|\,\|T_1\|=\|T_1\|\,\|T_2\|,
\end{align*}
we have $\lim_{n\rightarrow\infty} \|T_2^* T_1\xi_n\|=\|T_1\|\,\|T_2\|.$ Next observe that
\begin{multline*}
\|(T_2^* T_1-\lambda\|T_1\|\,\|T_2\|I)\xi_n\|^2=\|T_2^* T_1\xi_n\|^2-\overline{\lambda}\|T_1\|\,\|T_2\|(T_1\xi_n\mid T_2\xi_n)\\
-\lambda\|T_1\|\,\|T_2\|(T_2\xi_n\mid T_1\xi_n)+\|T_1\|^2\,\|T_2\|^2.
\end{multline*}
Therefore $\lim_{n\rightarrow\infty} \|(T_2^* T_1-\lambda\|T_1\|\,\|T_2\|I)\xi_n\|=0.$ Thus $r(T_2^* T_1)=\|T_1\|\,\|T_2\|=\|T_2^* T_1\|.$\\
The proof of the converse follows from the spectral inclusion theorem \cite[Theorem 1.2-1]{gu} that $\sigma(T_2^* T_1)\subseteq\overline{\{(T_2^* T_1\xi\mid \xi) :\,\xi\in\mathscr{H}, \|\xi\|=1\}}$, where the bar denotes the closure.\\
(ii)$\Longrightarrow$(iv) Let $\{\xi_n\}$ be a sequence of unit vectors in $\mathscr{H}$ which satisfies
$$\lim_{n\rightarrow\infty} (T_1\xi_n\mid T_2\xi_n)=\lambda\|T_1\|\,\|T_2\|,$$
for some $\lambda\in\mathbb{T}$. As in the proofs of the implications (ii)$\Longrightarrow$(i) and (ii)$\Longrightarrow$(iii), we get $\|T_1+\lambda T_2\|=\|T_1\|+\|T_2\|$ and $\|T_1^* T_2\|=\|T_1\|\,\|T_2\|$. By \cite[Theorem 3.3.6]{mor} there is a state $\varphi$ over $\mathbb{B}(\mathscr{H})$ such that
$$\varphi((T_1+\lambda T_2)^*(T_1+\lambda T_2))=\|(T_1+\lambda T_2)^*(T_1+\lambda T_2)\|=\|T_1+\lambda T_2\|^2=(\|T_1\|+\|T_2\|)^2.$$
Thus
\begin{align*}
(\|T_1\|+\|T_2\|)^2&=\varphi(T_1^* T_1+\lambda T_1^* T_2+\overline{\lambda}T_2^* T_1+T_2^* T_2)
\\&=\varphi(T_1^* T_1)+\varphi(\lambda T_1^* T_2+\overline{\lambda}T_2^* T_1)+\varphi(T_2^* T_2)
\\&\leq\|T_1^* T_1\|+\|\lambda T_1^* T_2+\overline{\lambda}T_2^* T_1\|+\|T_2^* T_2\|
\\&=\|T_1^* T_1\|+\|T_1^* T_2\|+\|T_2^* T_1\|+\|T_2^* T_2\|
\\&\leq\|T_1\|^2+2\|T_1\|\,\|T_2\|+\|T_2\|^2
\\&=(\|T_1\|+\|T_2\|)^2.
\end{align*}
Therefore $\varphi(T_1^* T_1)=\|T_1^* T_1\|$ and $\varphi(\lambda T_1^* T_2)=\|T_1^* T_2\|$. Hence
$$\|T_1^* T_1\|+\|T_1^* T_2\|=\varphi(T_1^* T_1+\lambda T_1^* T_2)\leq\|T_1^* T_1+\lambda T_1^* T_2\|\leq\|T_1^* T_1\|+\|T_1^* T_2\|.$$
Therefore $\|T_1^* T_1+\lambda T_1^* T_2\|=\|T_1^* T_1\|+\|T_1^* T_2\|$ for some $\lambda\in\mathbb{T}$. Thus $T_1^* T_1\parallel \lambda T_1^* T_2$.\\
(iv)$\Longrightarrow$(v) This implication is trivial.\\
(v)$\Longrightarrow$(i) Let $\|T_1^*(T_1+\lambda T_2)\|=\|T_1\|(\|T_1\|+\|T_2\|)$ for some $\lambda\in\mathbb{T}.$ Then we have
\begin{align*}
\|T_1\|(\|T_1\|+\|T_2\|)\geq\|T_1^*\|\|T_1+\lambda T_2\|
\geq\|T_1^*(T_1+\lambda T_2)\|
=\|T_1\|(\|T_1\|+\|T_2\|).
\end{align*}
Thus $\|T_1+\lambda T_2\|=\|T_1\|+\|T_2\|$, or equivalently, $T_1\parallel T_2$.
\end{proof}
As an immediate consequence of Theorem \ref{th.13}, we get a characterization of operator parallelism.
\begin{corollary}\label{co.14}
Let $T_1, T_2\in \mathbb{B}(\mathscr{H})$. Then the following statements are equivalent:\\
(i) $T_1\parallel T_2$;\\
(ii) ${T_i}^* T_i\parallel {T_j}^* T_i$ and $\|{T_j}^* T_i\|=\|T_j\|\,\|T_i\| \qquad (1\leq i\neq j \leq 2)$;\\
(iii) $T_i{T_i}^*\parallel T_i{T_j}^*$ and $\|T_i{T_j}^*\|=\|T_i\|\,\|T_j\| \qquad(1\leq i\neq j \leq 2)$.
\end{corollary}

\begin{corollary}\label{co.15}
Let $T_1, T_2\in \mathbb{B}(\mathscr{H})$. Then the following statements are equivalent:\\
(i) $T_1\parallel T_2$;\\
(ii) $r(T_2^* T_1)=\|T_2^* T_1\|=\|T_1\|\,\|T_2\|$;\\
(iii) $r(T_1T_2^*)=\|T_1T_2^*\|=\|T_1\|\,\|T_2\|$.
\end{corollary}

We need the next lemma for studying approximate parallelism.

\begin{lemma}\label{le.16}\cite[Proposition 2.1]{ar.2}
Let $T_1, T_2\in \mathbb{B}(\mathscr{H})$. Then $$\inf\{\|T_1+\mu T_2\|^2 :\, \mu\in\mathbb{C}\}=\sup\{M_{T_1, T_2}(\xi) :\, \xi\in\mathbb{C},\,\|\xi\|=1\},$$
where
$$M_{T_1, T_2}(\xi)=\begin{cases}
\|T_1\xi\|^2-\frac{|(T_1\xi\mid T_2\xi)|^2}{\|T_2\xi\|^2} &\text{if\, $T_2\xi\neq0$}\\
\|T_1\xi\|^2       &\text{if\, $T_2\xi=0$}.
\end{cases}$$
\end{lemma}
In the following proposition we present a characterization of the operator $\varepsilon$-parallelism $\parallel^{\varepsilon}$. Recall that for $T_1, T_2\in \mathbb{B}(\mathscr{H})$ and $\varepsilon\in[0 , 1)$, $T_1\parallel^{\varepsilon} T_2$ if
$\inf\{\|T_1+\mu T_2\| :\, \mu\in\mathbb{C}\}\leq\varepsilon\|T_1\|$.
\begin{theorem}\label{th.16}
Let $T_1, T_2\in \mathbb{B}(\mathscr{H})$ and $\varepsilon\in[0, 1)$. Then the following statements are equivalent:\\
(i) $T_1\parallel^{\varepsilon} T_2$;\\
(ii) $T_1^*\parallel^{\varepsilon} T_2^*$;\\
(iii) $\alpha T_1\parallel^{\varepsilon} \beta T_2$, \ \ \ \ \ $(\alpha, \beta\in\mathbb{C}\smallsetminus\{0\})$.\\
(iv) $\sup\{|(T_1\xi\mid \eta)|:\, \|\xi\|=\|\eta\|=1, \, (T_2\xi\mid \eta)=0\}\leq\varepsilon\|T_1\|.$\\
Moreover, each of the above conditions implies\\
(v) $|(T_1\xi\mid T_2\xi)|^2\geq\|T_1\xi\|^2\|T_2\xi\|^2-\varepsilon^2\|T_1\|^2\|T_2\|^2,\,\,\,\,\,\,\,(\xi\in\mathscr{H}, \|\xi\|=1).$
\end{theorem}
\begin{proof}
(i)$\Leftrightarrow$(ii) is obvious.\\
(i)$\Leftrightarrow$(iii) Let $T_1\parallel^{\varepsilon} T_2$ and $\alpha, \beta\in\mathbb{C}\smallsetminus\{0\}$. Then
\begin{align*}
\inf\{\|\alpha T_1+\mu(\beta T_2)\| :\, \mu\in\mathbb{C}\}&=|\alpha|\inf\{\|T_1+\frac{\mu\beta}{\alpha}T_2\| :\, \mu\in\mathbb{C}\}
\\&\leq |\alpha|\inf\{\|T_1+\nu T_2\| :\, \nu\in\mathbb{C}\}
\leq|\alpha|\varepsilon\|T_1\|=\varepsilon\|\alpha T_1\|.
\end{align*}
Therefore, $\alpha T_1\parallel^{\varepsilon} \beta T_2$. The converse is obvious.\\
(i)$\Leftrightarrow$(iv) Bhatia and \v{S}emrl \cite[Remark 3.1]{BE} proved that
$$\inf\{\|T_1+\mu T_2\| :\, \mu\in\mathbb{C}\}=\sup\{|(T_1\xi\mid \eta)|:\, \|\xi\|=\|\eta\|=1, \, (T_2\xi\mid \eta)=0\}$$
Thus the required equivalence follows from the above equality.\\
Now suppose that $T_1\parallel^{\varepsilon} T_2$. Hence $\inf\{\|T_1+\mu T_2\| :\, \mu\in\mathbb{C}\}\leq\varepsilon\|T_1\|$. For any $\xi\in \mathscr{H}$ with $\|\xi\|=1$, by Lemma \ref{le.16}, we therefore get
\begin{align*}
\|T_1\xi\|^2\|T_2\xi\|^2-|(T_1\xi\mid T_2\xi)|^2&\leq\|T_2\xi\|^2\inf\{\|T_1+\mu T_2\|^2 :\, \mu\in\mathbb{C}\}
\\&\leq\|T_2\xi\|^2\varepsilon^2\|T_1\|^2
\leq\varepsilon^2\|T_1\|^2\|T_2\|^2\|\xi\|^2
=\varepsilon^2\|T_1\|^2\|T_2\|^2.
\end{align*}
\end{proof}
In the following result we establish some equivalence statements to the approximate parallelism for elements of a Hilbert space. We use some techniques of \cite[Corollary 2.7]{Mos} to prove this corollary.
\begin{corollary}\label{co.17}
Let $\xi, \eta\in \mathscr{H}$. Then for any $\varepsilon\in[0 , 1)$ the following statements are equivalent:\\
(i) $\xi\parallel^{\varepsilon} \eta$;\\
(ii) $\sup\{|(\xi\mid \zeta)|:\,\zeta\in\mathscr{H}, \|\zeta\|=1, (\eta\mid \zeta)=0\}\leq\varepsilon\|\xi\|$;\\
(iii) $|(\xi\mid \eta)|\geq\sqrt{1-\varepsilon^2}\|\xi\|\,\|\eta\|$;\\
(v) $\Big\|\,\|\eta\|^2\xi-(\xi\mid \eta)\eta\Big\|\leq\varepsilon\|\xi\|\,\|\eta\|^2$.
\end{corollary}
\begin{proof}
Let $\psi$ be a unit vector of $\mathscr{H}$ and set $T_1=\xi\otimes \psi$ and $T_2=\eta\otimes \psi$ as rank one operators. A straightforward computation shows that $\xi\parallel^{\varepsilon} \eta$ if and only if $T_1\parallel^{\varepsilon} T_2$. It follows from the elementary properties of rank one operators and Lemma \ref{le.16} that\\
$M_{T_1, T_2}(\xi)=\begin{cases}
|(\xi\mid \psi)|^2\left(\|\xi\|^2-\frac{|(\xi\mid \eta)|^2}{\|\eta\|^2}\right) &\text{if $(\xi\mid \psi)\eta\neq0$}\\
|(\xi\mid \psi)|^2\|\xi\|^2       &\text{if $(\xi\mid \psi)\eta=0$}.
\end{cases}$\\
Thus we reach
\begin{align*}
\xi\parallel^{\varepsilon} \eta\,& \Longleftrightarrow \,T_1\parallel^{\varepsilon} T_2
\\& \Longleftrightarrow \, \sup\{M_{T_1, T_2}(\xi) :\, \xi\in\mathbb{C},\,\|\xi\|=1\}\leq\varepsilon^2\|T_1\|^2
\\& \Longleftrightarrow \, \|\xi\|^2\,\|\eta\|^2-|(\xi\mid \eta)|^2\leq\varepsilon^2\|\xi\|^2\|\eta\|^2
\\& \Longleftrightarrow \, |(\xi\mid \eta)|\geq\sqrt{1-\varepsilon^2}\|\xi\|\,\|\eta\|
\\& \Longleftrightarrow \, \Big\|\,\|\eta\|^2\xi-(\xi\mid \eta)\eta\Big\|\leq\varepsilon\|\xi\|\,\|\eta\|^2.
\end{align*}
Further, by the equivalence (i)$\Leftrightarrow$(iv) of Theorem \ref{th.16} yields the
\begin{align*}
\xi\parallel^{\varepsilon} \eta\,& \Longleftrightarrow \,T_1\parallel^{\varepsilon} T_2
\\& \Longleftrightarrow \, \sup\{|(T_1\omega\mid \zeta)|:\, \|\omega\|=\|\zeta\|=1, \, (T_2\omega\mid \zeta)=0\}\leq\varepsilon\|T_1\|
\\& \Longleftrightarrow \, \sup\{|(\omega\mid \psi)|\,|(\xi\mid \zeta)|:\, \|\omega\|=\|\zeta\|=1, \, (\omega\mid \psi)(\eta\mid \zeta)=0\}\leq\varepsilon\|\xi\|
\\& \Longleftrightarrow \, \sup\{|(\xi\mid \zeta)|:\,\zeta\in\mathscr{H}, \|\zeta\|=1, (\eta\mid \zeta)=0\}\leq\varepsilon\|\xi\|.
\end{align*}
\end{proof}

\begin{remark}
If we choose $\varepsilon=0$ in Corollary \ref{co.17}, we reach the fact that two vectors in a Hilbert space are parallel if and only if they are proportional.
\end{remark}

Next, we investigate the case when an operator is parallel to the identity operator.
\begin{theorem}\label{th.19}
Let $T\in \mathbb{B}(\mathscr{H})$. Then the following statements are equivalent:\\
(i) $T\parallel I$;\\
(ii) $T\parallel T^*$;\\
(iii) There exist a sequence of unit vectors $\{\xi_n\}$ in $\mathscr{H}$ and $\lambda\in\mathbb{T}$ such that $$\lim_{n\rightarrow\infty} \Big\|T\xi_n-\lambda\|T\|\xi_n\Big\|=0;$$
(iv) $T^m\parallel I\,\,\,\,\,\,(m\in\mathbb{N})$;\\
(v) $T^m\parallel {T^*}^m\,\,\,\,\,\,(m\in\mathbb{N})$.
\end{theorem}
\begin{proof}
(i)$\Leftrightarrow$(ii) Let $T\parallel I$. Then $\|T+\lambda I\|=\|T\|+1$ for some $\lambda\in\mathbb{T}$. By \cite[Theorem 3.3.6]{mor} there is a state $\varphi$ over $\mathbb{B}(\mathscr{H})$ such that
$$\varphi\left((T+\lambda I)(T+\lambda I)^*\right)=\|(T+\lambda I)(T+\lambda I)^*\|=\|T+\lambda I\|^2=(\|T\|+1)^2.$$
Thus
\begin{align*}
(\|T\|+1)^2&=\varphi((T+\lambda I)(T+\lambda I)^*)
=\varphi(TT^*)+\varphi(\overline{\lambda}T)+\varphi(\lambda T^*)+1
\\&\leq\|TT^*\|+\|\overline{\lambda}T\|+\|\lambda T^*\|+1
=\|T\|^2+2\|T\|+1=(\|T\|+1)^2.
\end{align*}
Therefore $\varphi(\overline{\lambda}T)=\varphi(\lambda T^*)=\|T\|$. This implies that
$$\|T\|+\|T^*\|=\varphi(\overline{\lambda}T+\lambda T^*)\leq\|\overline{\lambda}T+\lambda T^*\|=\|T+\lambda^2 T^*\|\leq\|T\|+\|T^*\|.$$
Therefore $\|T+\lambda^2 T^*\|=\|T\|+\|T^*\|$ in which $\lambda^2\in\mathbb{T}$. Thus $T\parallel T^*$.\\
To prove the converse, suppose that $T\parallel T^*$, or equivalently, $\|T+\lambda T^*\|=2\|T\|$ for some $\lambda\in\mathbb{T}$. By \cite[Theorem 3.3.6]{mor} there is a state $\varphi$ over $\mathbb{B}(\mathscr{H})$ such that $|\varphi(T+\lambda T^*)|=\|T+\lambda T^*\|=2\|T\|$. Thus we get
$2\|T\|=|\varphi(T+\lambda T^*)|\leq2|\varphi(T)|\leq2\|T\|$, from which it follows that $|\varphi(T)|=\|T\|$. Hence there exists a number $\mu\in\mathbb{T}$ such that $\varphi(T)=\mu\|T\|$. Therefore
$$\|T\|+1=\varphi(\overline{\mu}T+I)\leq\|\overline{\mu}T+I\|=\|T+\mu I\|\leq\|T\|+1,$$
whence $\|T+\mu I\|=\|T\|+1$ for $\mu\in\mathbb{T}$. Thus $T\parallel I$.\\
(i)$\Longleftrightarrow$(iii) let $T\parallel I$. By Theorem \ref{th.13}, there exist a sequence of unit vectors $\{\xi_n\}$ in $\mathscr{H}$ and $\lambda\in\mathbb{T}$ such that $\lim_{n\rightarrow\infty} (T\xi_n\mid \xi_n)=\lambda\|T\|$. Since $\|T\|=\lim_{n\rightarrow\infty} |(T\xi_n\mid \xi_n)|\leq\lim_{n\rightarrow\infty} \|T\xi_n\|\leq\|T\|$, hence $\lim_{n\rightarrow\infty} \|T\xi_n\|=\|T\|$. Thus
\begin{align*}
\lim_{n\rightarrow\infty} \Big\|T\xi_n-\lambda\|T\|\xi_n\Big\|^2&=\lim_{n\rightarrow\infty} \left[\|T\xi_n\|^2-\overline{\lambda}\|T\|(T\xi_n\mid \xi_n)-\lambda\|T\|(\xi_n\mid T\xi_n)+\|T\|^2\right]
\\&=\|T\|^2-|\lambda|^2\|T\|^2-|\lambda|^2\|T\|^2+\|T\|^2=0.
\end{align*}
So that $\lim_{n\rightarrow\infty} \Big\|T\xi_n-\lambda\|T\|\xi_n\Big\|=0$.\\
Conversely, suppose that (iii) is holds. Then
\begin{align*}
1+\|T\|\geq\|T+\lambda I\|&\geq\|T\xi_n+\lambda \xi_n\|=\Big\|\lambda \xi_n+\lambda\|T\|\xi_n-(-T\xi_n+\lambda\|T\|\xi_n)\Big\|
\\&\geq\Big\|\lambda \xi_n+\lambda\|T\|\xi_n\Big\|-\Big\|-T\xi_n+\lambda\|T\|\xi_n\Big\|
\\&=1+\|T\|-\Big\|T\xi_n-\lambda\|T\|\xi_n\Big\|.
\end{align*}
By taking limits, we get
$$1+\|T\|\geq\|T+\lambda I\|\geq1+\|T\|,$$
so $\|T+\lambda I\|=1+\|T\|$, i.e, $T\parallel I$.\\
(iii)$\Longrightarrow$(iv) Let there exists a sequence of unit vectors $\{\xi_n\}$ in $\mathscr{H}$ and $\lambda\in\mathbb{T}$ such that $\lim_{n\rightarrow\infty} \Big\|T\xi_n-\lambda\|T\|\xi_n\Big\|=0$.
For any $k\in\mathbb{N}$ we have
\begin{align*}
\Big\|(T^{k+1}-\lambda^{k+1}\|T\|^{k+1}I)\xi_n\Big\|&=\Big\|T(T^{k}-\lambda^{k}\|T\|^{k}I)\xi_n+\lambda^{k}\|T\|^{k}(T-\lambda\|T\|I)\xi_n\Big\|
\\&\leq \|T\|\,\Big\|(T^{k}-\lambda^{k}\|T\|^{k}I)\xi_n\Big\|+\|T\|^k\,\Big\|(T-\lambda\|T\|I)\xi_n\Big\|
\end{align*}
Hence, by induction, we have $$\lim_{n\rightarrow\infty} \Big\|(T^m-\lambda^m\|T\|^mI)\xi_n\Big\|=0$$
for all $m\in\mathbb{N}$. We get $\|T\|^m\leq r(T^m)\leq\|T^m\|\leq\|T\|^m$. Hence $\|T\|^m=\|T^m\|$.
Now for $\mu=\lambda^m\in\mathbb{T}$ we have $$\lim_{n\rightarrow\infty} \Big\|T^m\xi_n-\mu\|T^m\|\xi_n\Big\|=\lim_{n\rightarrow\infty} \Big\|(T^m-\lambda^m\|T\|^mI)\xi_n\Big\|=0.$$
So by the equivalence (i)$\Leftrightarrow$(iii), we get $T^m\parallel I$.\\
The implications (iv)$\Longrightarrow$(v) and (v)$\Longrightarrow$(i), follow from the equivalence (i)$\Leftrightarrow$(ii).
\end{proof}
For $T\in \mathbb{B}(\mathscr{H})$ the operator $\delta_T(S)=TS-ST$ over $\mathbb{B}(\mathscr{H})$ is called an \emph{inner derivation}. Clearly $2\|T\|$ is a upper bound for $\|\delta_T\|$.
In the next result, we get a characterization of operator $\varepsilon$-parallelism.
\begin{corollary}\label{co.120}
Let $T\in \mathbb{B}(\mathscr{H})$ and $\varepsilon\in[0, 1)$. The following statements are equivalent:\\
(i) $T\parallel^{\varepsilon} I$;\\
(ii) $\sup\{\|T\xi-(T\xi\mid \xi)\xi\|:\, \|\xi\|=1\}\leq\varepsilon\|T\|$; \\
(iii) $\sup\{\|T\xi\|^2-|(T\xi\mid \xi)|^2:\, \|\xi\|=1\}\leq\varepsilon^2\|T\|^2$;\\
(iv) $\|\delta_T\|\leq2\varepsilon\|T\|$.
\end{corollary}
\begin{proof}
Fujii and Nakamoto \cite{FU} proved that
\begin{align*}
\left(\sup\{\|T\xi\|^2-|(T\xi\mid \xi)|^2:\, \|\xi\|=1\}\right)^\frac{1}{2}&=\inf\{\|T+\mu I\| :\, \mu\in\mathbb{C}\}
\\&=\sup\{\|T\xi-(T\xi\mid \xi)\xi\|:\, \|\xi\|=1\}.
\end{align*}
Thus the implications (i)$\Longrightarrow$(ii) and (ii)$\Longrightarrow$(iii) follow immediately from the above identities.\\
On the other hand by \cite[Remark 3.2]{BE} we have $$\sup\{\|TS-ST\|:\, \|S\|=1\}=2\inf\{\|T+\mu I\| :\, \mu\in\mathbb{C}\}.$$
Therefore we get $T\parallel^{\varepsilon} I$ if and only if $\|\delta_T\|=\sup\{\|TS-ST\|:\, \|S\|=1\}\leq2\varepsilon\|T\|$.
\end{proof}
Two operators $T_1, T_2\in \mathbb{B}(\mathscr{H})$ are unitarily equivalent if there exists a unitary operator $S$ such that $S^*T_1S=T_2.$ Clearly $\|T_1\|=\|T_2\|.$
\begin{proposition}\label{pr.121}
Let $T_1, T_2\in \mathbb{B}(\mathscr{H})$ be unitarily equivalent and $\varepsilon\in[0, 1)$. Then \\
(i) $T_1\parallel I\Longleftrightarrow T_2\parallel I$.\\
(ii) $T_1\parallel^{\varepsilon} I\Longleftrightarrow T_2\parallel^{\varepsilon} I$.
\end{proposition}
\begin{proof}
(i) Since $T_1, T_2\in \mathbb{B}(\mathscr{H})$ are unitarily equivalent, there exists a unitary operator $S$ such that $S^*T_1S=T_2.$ Then\\
\begin{align*}
T_1\parallel I&\Longleftrightarrow \|T_1+\lambda I\|=\|T_1\|+\|I\|\,\,\,\,\,\,\,\,\,\, for\,\, some\,\, \lambda\in\mathbb{T}
\\&\Longleftrightarrow \|S^*(T_1+\lambda I)S\|=\|S^*T_1S\|+\|S^*IS\|\,\,\,\,\,\,\,\,\,\, for\,\, some\,\, \lambda\in\mathbb{T}
\\&\Longleftrightarrow \|T_2+\lambda I\|=\|T_2\|+\|I\|\,\,\,\,\,\,\,\,\,\, for\,\, some\,\, \lambda\in\mathbb{T}
\\&\Longleftrightarrow T_2\parallel I.
\end{align*}
(ii) It can be proved by the same reasoning as in the proof of (i).
\end{proof}
We finish this section with an application of the concept $\varepsilon$-parallelism to some special types of elementary operators. We state some prerequisites for the next result. Let $\mathscr{V}$ be a normed space and $\mathbb{B}(\mathscr{V})$ denotes the algebra of the bounded linear operators on $\mathscr{V}$. A standard operator algebra $\mathfrak{B}$ is a subalgebra of $\mathbb{B}(\mathscr{V})$ that contains all finite rank operators on $\mathscr{V}$. For $T_1, T_2\in\mathfrak{B}$ we denote $M_{T_1, T_2}$, $V_{T_1, T_2}$ and $U_{T_1, T_2}$ on $\mathfrak{B}$ by $M_{T_1, T_2}(S)=T_1ST_2$, $V_{T_1, T_2}=M_{T_1,T_2}-M_{T_2, T_1}$ and $U_{T_1, T_2}=M_{T_1, T_2}+M_{T_2, T_1}$, for every $S\in\mathfrak{B}.$ We denote by $d(U_{T_1, T_2})$ the supremum of the norm of $U_{T_1, T_2}(S)$ over all rank one operators of norm one on $\mathscr{V}$. Similarly $d(M_{T_1, T_2})$ and $d(V_{T_1, T_2})$ are defined. It is easy to see that $d(M_{T_1, T_2})=\|M_{T_1, T_2}\|=\|T_1\|\,\|T_2\|$ and $V_{T_1+\mu T_2, T_2}=V_{T_1, T_2}$ for all scalar $\mu$. To establish the following proposition we use some ideas of \cite[Theorem 11]{se.2}.

\begin{proposition}\label{pr.122}
Let $\mathfrak{B}$ be a standard operator algebra and $T_1, T_2\in\mathfrak{B}$. Then the estimate $d(U_{T_1, T_2})\geq2(1-\varepsilon)\|T_1\|\,\|T_2\|$ holds if one of the following properties is satisfied:\\
(i) $T_1\parallel^{\varepsilon} T_2;$\\
(ii) $T_2\parallel^{\varepsilon} T_1.$
\end{proposition}
\begin{proof}
Let $T_1\parallel^{\varepsilon} T_2$. Hence $\inf\{\|T_1+\mu T_2\| :\, \mu\in\mathbb{C}\}\leq\varepsilon\|T_1\|.$ For every $\mu\in\mathbb{C}$ we have
\begin{align*}
\|V_{T_1, T_2}\|=\|V_{T_1+\mu T_2, T_2}\|&=\|M_{T_1+\mu T_2, T_2}-M_{T_2, T_1+\mu T_2}\|
\\&\leq\|M_{T_1+\mu T_2, T_2}\|+\|M_{T_2, T_1+\mu T_2}\|
=2\|T_2\|\,\|T_1+\mu T_2\|.
\end{align*}
Hence $$\|V_{T_1, T_2}\|\leq2\|T_2\|\inf\{\|T_1+\mu T_2\| :\, \mu\in\mathbb{C}\}\leq2\varepsilon\|T_1\|\,\|T_2\|,$$
from which we get $$d(V_{T_1, T_2})\leq2\varepsilon\|T_1\|\,\|T_2\|.$$
It follows from $U_{T_1, T_2}=2M_{T_1, T_2}-V_{T_1, T_2}$ that $$d(U_{T_1, T_2})\geq2d(M_{T_1, T_2})-d(V_{T_1, T_2})\geq2\|T_1\|\,\|T_2\|-2\varepsilon\|T_1\|\,\|T_2\|=2(1-\varepsilon)\|T_1\|\,\|T_2\|.$$
By the same argument, the estimation follows under the condition (ii).
\end{proof}
\section{Parallelism in $C^*$-algebras and inner product $C^*$-modules}
The relations between parallel elements in Hilbert $C^*$-modules form the main topic of this section.
We describe the concept parallelism in Hilbert $C^*$-modules. The notion of state plays an important role in this investigation. We begin with the following proposition, which will be useful
in other contexts as well. In this theorem we establish some equivalent assertions about the parallelism of elements of a Hilbert $C^*$-module. The proofs of implication (i)$\Rightarrow$(ii) in Theorem \ref{th.21} and Corollary \ref{co.22} are modification of ones given by Aramba\v{s}i\'{c} and Raji\'{c} \cite[Theorem 2.1, 2.9]{ar.1}. We present the proof for the sake of completeness.

\begin{theorem}\label{th.21}
Let $\mathscr{X}$ be a Hilbert $C^*$-module over a $C^*$-algebra ${\mathscr A}$. For $x , y\in \mathscr{X}$ the following statements are equivalent:\\
(i) $x\parallel y$;\\
(ii) There exist a state $\varphi$ over $\mathscr A$ and $\lambda\in\mathbb{T}$ such that $\varphi({\langle x, y\rangle})=\lambda\|x\|\|y\|;$\\
(iii) There exist a norm one linear functional $f$ over $\mathscr{X}$ and $\lambda\in\mathbb{T}$ such that $f(x) = \|x\|$ and $f(y) = \lambda\|y\|.$
\end{theorem}
\begin{proof}
(i)$\Rightarrow$(ii) Let $x\parallel y$. Hence $\|x+\overline{\lambda} y\|=\|x\|+\|y\|$ for some $\lambda\in\mathbb{T}$. By \cite[Theorem 3.3.6]{mor} there is a state $\varphi$ over $\mathscr A$ such that $$\varphi({\langle x+\overline{\lambda} y, x+\overline{\lambda} y\rangle})=\|{\langle x+\overline{\lambda} y, x+\overline{\lambda} y\rangle}\|=\|x+\overline{\lambda} y\|^2.$$ We therefore have
\begin{align*}
\|x+\overline{\lambda} y\|^2=&\varphi({\langle x+\overline{\lambda} y, x+\overline{\lambda} y\rangle})
\\&=\varphi({\langle x, x\rangle})+\varphi({\langle x, \overline{\lambda} y\rangle})+\varphi({\langle \overline{\lambda} y, x\rangle})+|\lambda|^2\varphi({\langle y, y\rangle})
\\&=\varphi({\langle x, x\rangle})+2\mbox{Re}\varphi({\langle x, \overline{\lambda} y\rangle})+\varphi({\langle y, y\rangle})
\\&\leq\|x\|^2+2\|{\langle x, \overline{\lambda} y\rangle}\|+\|y\|^2
\\&\leq\|x\|^2+2\|x\|\|y\|+\|y\|^2
\\&=(\|x\|+\|y\|)^2=\|x+\overline{\lambda} y\|^2.
\end{align*}
Thus we get $\varphi({\langle x, x\rangle})=\|x\|^2$, $\varphi({\langle y, y\rangle})=\|y\|^2$ and $\varphi({\langle x, y\rangle})=\lambda\|x\|\|y\|.$\\
(ii)$\Rightarrow$(iii) Suppose that there exist a state $\varphi$ over $\mathscr A$ and $\lambda\in\mathbb{T}$ such that $\varphi({\langle x, y\rangle})=\lambda\|x\|\|y\|.$ We may assume that $x\neq0$. Define a linear functional $f$ on $\mathscr{X}$ by $$f(z)=\frac{\varphi({\langle x, z\rangle})}{\|x\|}\,\,\,\,\,\,\,\,\,\,(z\in E).$$
It follows from $$|f(z)|=\left|\frac{\varphi({\langle x, z\rangle})}{\|x\|}\right|\leq\frac{\|{\langle x, z\rangle}\|}{\|x\|}\leq\|z\|,$$ that $\|f\|\leq1$.
We infer from the Cauchy--Schwarz inequality that
$$\|x\|^2\|y\|^2=|\varphi({\langle x, y\rangle})|^2\leq\varphi({\langle x, x\rangle})\varphi({\langle y, y\rangle})\leq\|x\|^2\|y\|^2,$$
so $\varphi({\langle x, x\rangle})=\|x\|^2$ and hence $f(x)=\frac{\varphi({\langle x, x\rangle})}{\|x\|}=\frac{\|x\|^2}{\|x\|}=\|x\|$. Thus $\|f\|=1$ and $f(y)=\frac{\varphi({\langle x, y\rangle})}{\|x\|}=\frac{\lambda\|x\|\|y\|}{\|x\|}=\lambda\|y\|.$\\
(iii)$\Rightarrow$(i) Suppose that there exist a norm one linear functional $f$ over $\mathscr{X}$ and $\lambda\in\mathbb{T}$ such that $f(x) =\|x\|$ and $f(y) = \lambda\|y\|$. Hence
$$\|x\|+\|y\|=f(x)+f(\overline{\lambda}y)=f(x+\overline{\lambda}y)\leq\|x+\overline{\lambda}y\|\leq\|x\|+
\|\overline{\lambda}y\|=\|x\|+\|y\|.$$
So, we have $\|x+\overline{\lambda}y\|=\|x\|+\|y\|$ for $\overline{\lambda}\in\mathbb{T}$. Thus $x\parallel y$.
\end{proof}

\begin{corollary} \label{co.22}
Let $\mathscr{X}$ be a Hilbert ${\mathscr A}$-module and $x , y\in \mathscr{X}\smallsetminus\{0\}$.\\
(i) If $x\parallel y$, then there exist a state $\varphi$ over $\mathscr A$ and $\lambda\in\mathbb{T}$ such that $$\frac{\|y\|}{\|x\|}\varphi(|x|^2)+\frac{\|x\|}{\|y\|}\varphi(|y|^2)=2\lambda\varphi({\langle x, y\rangle}).$$
(ii) Let ${\mathscr A}$ has an identity $e$. If either $|x|^2=e$ or $|y|^2=e$ and there exist a state $\varphi$ over $\mathscr A$ and $\lambda\in\mathbb{T}$ such that $\frac{\|y\|}{\|x\|}\varphi(|x|^2)+\frac{\|x\|}{\|y\|}\varphi(|y|^2)=2\lambda\varphi({\langle x, y\rangle})$, then $x\parallel y$.
\end{corollary}
\begin{proof}
(i) Let $x\parallel y$. As in the proof of Theorem \ref{th.21}, there exist a state $\varphi$ over $\mathscr A$ and $\lambda\in\mathbb{T}$ such that $\varphi(|x|^2)=\varphi({\langle x, x\rangle})=\|x\|^2$, $\varphi(|y|^2)=\varphi({\langle y, y\rangle})=\|y\|^2$ and $\varphi({\langle x, \lambda y\rangle})=\|x\|\|y\|.$ Thus
$$\frac{\|y\|}{\|x\|}\varphi(|x|^2)+\frac{\|x\|}{\|y\|}\varphi(|y|^2)=\frac{\|y\|}{\|x\|}\cdot\|x\|^2+\frac{\|x\|}{\|y\|}\|y\|^2=2\|x\|\|y\|=2\lambda\varphi({\langle x, y\rangle}).$$
(ii) We may assume that $|x|^2=e$. We have
\begin{align*}
0\leq\left(\sqrt{\|y\|}-\sqrt{\frac{\varphi(|y|^2)}{\|y\|}}\right)^2&
=\left(\sqrt{\frac{\|y\|}{\|x\|}\varphi(|x|^2)}-\sqrt{\frac{\|x\|}{\|y\|}\varphi(|y|^2)}\right)^2
\\&=\frac{\|y\|}{\|x\|}\varphi(|x|^2)+\frac{\|x\|}{\|y\|}\varphi(|y|^2)-2\sqrt{\varphi(|x|^2)\varphi(|y|^2)}
\\&=\frac{\|y\|}{\|x\|}\varphi(|x|^2)+\frac{\|x\|}{\|y\|}\varphi(|y|^2)-2\sqrt{\varphi({\langle x, x\rangle})\varphi({\langle \lambda y, \lambda y\rangle})}
\\&\leq\frac{\|y\|}{\|x\|}\varphi(|x|^2)+\frac{\|x\|}{\|y\|}\varphi(|y|^2)-2\sqrt{|\varphi(\langle x, \lambda y\rangle)|^2}
\\&\hspace{2.7cm}(\mbox{by the Cauchy--Schwarz inequality})
\\&= \frac{\|y\|}{\|x\|}\varphi(|x|^2)+\frac{\|x\|}{\|y\|}\varphi(|y|^2)-| 2\lambda\varphi(\langle x, y\rangle)|
\\&= \frac{\|y\|}{\|x\|}\varphi(|x|^2)+\frac{\|x\|}{\|y\|}\varphi(|y|^2)-\left|\frac{\|y\|}{\|x\|}\varphi(|x|^2)+\frac{\|x\|}{\|y\|}\varphi(|y|^2)\right|
\\&=0.
\end{align*}
We conclude that $\varphi({\langle y, y\rangle})=\|y\|^2$ and $\varphi({\langle x, y\rangle})=\overline{\lambda}\sqrt{\varphi(|x|^2)\varphi(|y|^2)}=\overline{\lambda}\|x\|\|y\|$, since $2\lambda\varphi(\langle x, y\rangle)\geq 0$. Thus, by Theorem \ref{th.21} (ii), we get $x\parallel y$.
\end{proof}

\begin{corollary}\label{co.23}
Let $\mathscr{X}$ be a Hilbert ${\mathscr A}$-module and $x , y\in \mathscr{X}$. Then the following statements are equivalent:\\
(i) $x\parallel y$;\\
(ii) There exist a state $\varphi$ over $\mathbb{K}(\mathscr{X})$ and $\lambda\in\mathbb{T}$ such that $\varphi(\theta_{x, y})=\lambda\|x\|\|y\|.$
\end{corollary}
\begin{proof}
Since $\mathscr{X}$ can be regarded as a left Hilbert $\mathbb{K}(\mathscr{X})$-module via the inner product $[x, y]=\theta_{x, y}$, therefore we reach the result by using Theorem \ref{th.21}.
\end{proof}
The following result characterizes the parallelism for elements of a $C^*$-algebra.
\begin{corollary}\label{co.25}
Let ${\mathscr A}$ be a $C^*$-algebra, and $a, b\in {\mathscr A}$. Then the following statements are equivalent:\\
(i) $a\parallel b$;\\
(ii) There exist a state $\varphi$ over $\mathscr A$ and $\lambda\in\mathbb{T}$, such that $\varphi(a^*b)=\lambda\|a\|\|b\|$;\\
(iii) There exist a Hilbert space $\mathscr{H}$, a representation $\pi:{\mathscr A}\to \mathbb{B}(\mathscr{H})$, a unit vector $\xi\in \mathscr{H}$ and $\lambda\in\mathbb{T}$ such that $\|\pi(a)\xi\|=\|a\|$ and $(\pi(a)\xi\mid \pi(b)\xi)=\lambda\|a\|\|b\|.$
\end{corollary}
\begin{proof}
If ${\mathscr A}$ is regarded as a Hilbert ${\mathscr A}$-module then the equivalence (i)$\Longleftrightarrow$(ii) follows from Theorem \ref{th.21}.\\
To show (ii)$\Longrightarrow$(iii), suppose that there are a state $\varphi$ and $\lambda\in\mathbb{T}$ such that $\varphi(a^*b)=\lambda\|a\|\|b\|$. By the Cauchy--Schwarz inequality we have
$$\|a\|^2\|b\|^2=|\varphi(a^*b)|^2\leq\varphi(a^*a)\varphi(b^*b)\leq\|a\|^2\|b\|^2,$$
so $\varphi(a^*a)=\|a\|^2$. By \cite[Proposition 2.4.4]{dix} there exist a Hilbert space $\mathscr{H}$, a representation $\pi: {\mathscr A}\to \mathbb{B}(\mathscr{H})$ and a unit vector $\xi\in \mathscr{H}$ such that for any $c\in{\mathscr A}$ we have $\varphi(c)=(\pi(c)\xi\mid \xi).$ Hence $$\|\pi(a)\xi\|=\sqrt{(\pi(a)\xi\mid \pi(a)\xi)}=\sqrt{(\pi(a^* a)\xi\mid \xi)}=\sqrt{\varphi(a^*a)}=\|a\|,$$
and $$(\pi(b)\xi\mid \pi(a)\xi)=(\pi(a^* b)\xi\mid \xi)=\varphi(a^*b)=\lambda\|a\|\|b\|.$$
Finally, we show (iii)$\Rightarrow$(ii). Let condition (iii) holds and let $\varphi:{\mathscr A}\to \mathbb{C}$ be the state associated to $\pi$ and $\xi$ by $\varphi(c)=(\pi(c)\xi\mid \xi)$, $c\in {\mathscr A}$. Thus
$$\varphi(a^*b)=(\pi(a^* b)\xi\mid \xi)=(\pi(b)\xi\mid \pi(a)\xi)=\lambda\|a\|\|b\|.$$
\end{proof}
The proof of the following proposition is a modification of one given by Rieffel \cite[Theorem 3.10]{Re}.
\begin{proposition}\label{pr.24}
Let ${\mathscr A}$ be a $C^*$-algebra with identity $e$ and $\varepsilon\in[0, 1)$. Then for any $a\in {\mathscr A}$ the following statements are equivalent:\\
(i) $a\parallel^{\varepsilon} e$;\\
(ii) $\max\{\sqrt{\varphi(a^* a)-|\varphi(a)|^2}:\, \varphi\in S({\mathscr A})\}\leq\varepsilon\|a\|$
\end{proposition}
\begin{proof}
(i)$\Longrightarrow$(ii) For every $\varphi\in S({\mathscr A})$ and $\mu\in\mathbb{C}$ a direct calculation shows that
\begin{align*}
\sqrt{\varphi(a^* a)-|\varphi(a)|^2}&=\sqrt{\varphi\Big((a+\mu e)^* (a+\mu e)\Big)-|\varphi(a+\mu e)|^2}
\\&\leq\sqrt{\varphi\Big((a+\mu e)^* (a+\mu e)\Big)}
\leq\|a+\mu e\|.
\end{align*}
So $\max\{\sqrt{\varphi(a^* a)-|\varphi(a)|^2}:\, \varphi\in S({\mathscr A})\}\leq\inf\{\|a+\mu e\| :\, \mu\in\mathbb{C}\}$. Since $a\parallel^{\varepsilon} e$, hence $\inf\{\|a+\mu e\| :\, \mu\in\mathbb{C}\}\leq\varepsilon\|a\|$. Thus $$\max\{\sqrt{\varphi(a^* a)-|\varphi(a)|^2}:\, \varphi\in S({\mathscr A})\}\leq\varepsilon\|a\|.$$
(ii)$\Longrightarrow$(i) Let (ii) holds and let $\inf\{\|a+\mu e\| :\, \mu\in\mathbb{C}\}=\|a+\alpha e\|$ for some $\alpha\in\mathbb{C}$. Then for any $\mu\in\mathbb{C}$ we have $\|(a+\alpha e)+\mu e\|\geq\|a+\alpha e\|$, whence by \cite[Theorem 2.7]{ar.2}, there exists a state $\varphi_\alpha\in S({\mathscr A})$ such that $$\sqrt{\varphi_\alpha\Big((a+\alpha e)^* (a+\alpha e)\Big)}=\|a+\alpha e\| \quad \mbox{and} \quad \varphi_\alpha(a)=-\alpha.$$
Therefore
\begin{align*}
\inf\{\|a+\mu e\| :\, \mu\in\mathbb{C}\}&=\|a+\alpha e\|
=\sqrt{\varphi_\alpha\Big((a+\alpha e)^* (a+\alpha e)\Big)}
\\&=\sqrt{\varphi_\alpha(a^* a)+\overline{\alpha}\varphi_\alpha(a)+\alpha\varphi_\alpha(a^*)+|\alpha|^2}
=\sqrt{\varphi_\alpha(a^* a)-|\varphi_\alpha(a)|^2}
\\&\leq\max\{\sqrt{\varphi(a^* a)-|\varphi(a)|^2}:\, \varphi\in S({\mathscr A})\}
\end{align*}
Thus $\inf\{\|a+\mu e\| :\, \mu\in\mathbb{C}\}\leq\varepsilon\|a\|$, or equivalently, $a\parallel^{\varepsilon} e$.
\end{proof}
In the following result, we utilize the linking algebra to give some equivalence assertions regarding parallel elements in a Hilbert $C^*$-module.
\begin{theorem}\label{th.22}
Let $\mathscr{X}$ be a Hilbert ${\mathscr A}$-module and $x , y\in \mathscr{X}$. Then the following statements are mutually equivalent:\\
(i) $x\parallel y$;\\
(ii) ${\langle x, x\rangle}\parallel{\langle x, y\rangle}$ and $\|{\langle x, y\rangle}\|=\|x\|\|y\|;$\\
(iii) $r({\langle x, y\rangle})=\|{\langle x, y\rangle}\|=\|x\|\|y\|;$\\
(iv) $\|{\langle x, x+\lambda y\rangle}\|=\|x\|(\|x\|+\|y\|)$ for some $\lambda\in\mathbb{T}$.
\end{theorem}
\begin{proof}
Consider the elements
$\begin{bmatrix}
0 & 0 \\
r_x & 0
\end{bmatrix}$
and
$\begin{bmatrix}
0 & 0 \\
r_y & 0
\end{bmatrix}$
of the $C^*$-algebra $\mathbb{L}(\mathscr{X})$, the linking algebra of $\mathscr{X}$. Let $\pi: \mathbb{L}(\mathscr{X})\to \mathbb{B}(\mathscr{H})$ be a non-degenerate faithful representation of $\mathbb{L}(\mathscr{X})$ on some Hilbert space $\mathscr{H}$ \cite[Theorem 2.6.1]{dix}.\\
(i)$\Longleftrightarrow$(ii) A straightforward computation shows that
$$x\parallel y \,\Longleftrightarrow\, \begin{bmatrix}
0 & 0 \\
r_x & 0
\end{bmatrix} \,\Big\| \,\begin{bmatrix}
0 & 0 \\
r_y & 0
\end{bmatrix} \,\Longleftrightarrow\, \pi\left(\begin{bmatrix}
0 & 0 \\
r_x & 0
\end{bmatrix}\right) \,\Big\| \,\pi\left(\begin{bmatrix}
0 & 0 \\
r_y & 0
\end{bmatrix}\right).$$
Thus by Theorem \ref{th.13}, we get
\begin{align*}
x\parallel y &\,\Longleftrightarrow\, \pi\left(\begin{bmatrix}
0 & 0 \\
r_x & 0
\end{bmatrix}\right) \,\Big\| \,\pi\left(\begin{bmatrix}
0 & 0 \\
r_y & 0
\end{bmatrix}\right)
\\&\,\Longleftrightarrow\,\pi\left(\begin{bmatrix}
0 & 0 \\
r_x & 0
\end{bmatrix}\right)^*\pi\left(\begin{bmatrix}
0 & 0 \\
r_x & 0
\end{bmatrix}\right) \,\Big\| \,\pi\left(\begin{bmatrix}
0 & 0 \\
r_x & 0
\end{bmatrix}\right)^*\pi\left(\begin{bmatrix}
0 & 0 \\
r_y & 0
\end{bmatrix}\right)
\\&\qquad\mbox{and}\,\,\left\|\pi\left(\begin{bmatrix}
0 & 0 \\
r_x & 0
\end{bmatrix}\right)^*\pi\left(\begin{bmatrix}
0 & 0 \\
r_y & 0
\end{bmatrix}\right)\right\|=\left\|\pi\left(\begin{bmatrix}
0 & 0 \\
r_x & 0
\end{bmatrix}\right)\right\|\,\left\|\pi\left(\begin{bmatrix}
0 & 0 \\
r_y & 0
\end{bmatrix}\right)\right\|
\\&\,\Longleftrightarrow\,\pi\left(\begin{bmatrix}
0 & l_xr_x \\
0 & 0
\end{bmatrix}\right) \,\Big\| \,\pi\left(\begin{bmatrix}
0 & l_xr_y \\
0 & 0
\end{bmatrix}\right)
\\&\qquad\mbox{and}\,\,\left\|\pi\left(\begin{bmatrix}
0 & l_xr_x \\
0 & 0
\end{bmatrix}\right)\right\|=\left\|\pi\left(\begin{bmatrix}
0 & l_x \\
0 & 0
\end{bmatrix}\right)\right\|\,\left\|\pi\left(\begin{bmatrix}
0 & 0 \\
r_y & 0
\end{bmatrix}\right)\right\|
\\&\,\Longleftrightarrow\,\begin{bmatrix}
0 & T_{{\langle x, x\rangle}} \\
0 & 0
\end{bmatrix} \,\Big\| \,\begin{bmatrix}
0 & T_{{\langle x, y\rangle}} \\
0 & 0
\end{bmatrix}\quad\mbox{and}\,\,\left\|\begin{bmatrix}
0 & T_{{\langle x, y\rangle}} \\
0 & 0
\end{bmatrix}\right\|=\left\|\begin{bmatrix}
0 & l_x \\
0 & 0
\end{bmatrix}\right\|\,\left\|\begin{bmatrix}
0 & 0 \\
r_y & 0
\end{bmatrix}\right\|
\\&\,\Longleftrightarrow\,{\langle x, x\rangle}\parallel{\langle x, y\rangle}\quad\mbox{and}\,\,\|{\langle x, y\rangle}\|=\|x\|\,\|y\|.
\end{align*}
(ii)$\Longleftrightarrow$(iii) By the equivalence (iii)$\Longleftrightarrow$(iv) of Theorem \ref{th.13}, the proof is similar to the proof of the equivalence (i)$\Longleftrightarrow$(ii), so we omit it.\\
(ii)$\Longrightarrow$(iv) Since ${\langle x, x\rangle}\parallel{\langle x, y\rangle}$, we have $\|{\langle x, x\rangle}+ \lambda{\langle x, y\rangle}\|= \|{\langle x, x\rangle}\|+\|{\langle x, y\rangle}\|$ for some $\lambda\in\mathbb{T}$. It follows from  $\|{\langle x, y\rangle}\|=\|x\|\|y\|$ that
$$\|{\langle x, x+\lambda y\rangle}\|=\|{\langle x, x\rangle}+ \lambda{\langle x, y\rangle}\|= \|{\langle x, x\rangle}\|+\|{\langle x, y\rangle}\|=\|x\|(\|x\|+\|y\|).$$
(iv)$\Longrightarrow$(i) We may assume that $x\neq 0$. Due to $\|{\langle x, x+\lambda y\rangle}\|=\|x\|(\|x\|+\|y\|)$ for some $\lambda\in\mathbb{T}$, by the Cauchy--Schwarz inequality, we have
\begin{align*}
\|x\|(\|x\|+\|y\|)=\|{\langle x, x+\lambda y\rangle}\|\leq\|x\|\|x+\lambda y\|\leq\|x\|(\|x\|+\|y\|).
\end{align*}
Thus $\|x+\lambda y\|=\|x\|+\|y\|$. Hence $x\parallel y$.
\end{proof}
Now, by Theorem \ref{th.16} and the same technique used to prove Theorem \ref{th.22} the final result is obtained.
\begin{corollary}\label{co.25}
Let $\mathscr{X}$ be a Hilbert ${\mathscr A}$-module, $x, y\in \mathscr{X}$ and $\varepsilon\in[0, 1)$. If $x\parallel^{\varepsilon} y$ then
$$|\varphi({\langle x, y\rangle})|^2\geq\varphi({\langle x, x\rangle})\varphi({\langle y, y\rangle})-\varepsilon^2\|{\langle x, x\rangle}\|\|{\langle y, y\rangle}\|\,\,\,\,\,\,\,(\varphi\in S({\mathscr A})).$$
\end{corollary}

\bibliographystyle{amsplain}

\end{document}